\documentclass[12pt,a4paper]{article}
\usepackage{graphicx}
\usepackage{amsmath}
\usepackage[cp1251]{inputenc}
\usepackage[T2A]{fontenc}
\usepackage{a4}
\usepackage{amsfonts}
\usepackage{amssymb}

\newtheorem{theorem}{Theorem}

\newtheorem{corollary}[theorem]{Corollary}

\newtheorem{lemma}[theorem]{Lemma}

\newtheorem{remark}[theorem]{Remark}

\newenvironment{proof}[1][Proof]{\textbf{#1.} }{\ \rule{0.5em}{0.5em}}
\numberwithin{equation}{section}

\begin{document}

\bigskip


\bigskip

\textbf{Liouville property for solutions of the linearized
degenerate thin film equation of fourth order in a halfspace.}

\bigskip

S.P.Degtyarev

\bigskip

\textbf{Institute for applied mathematics and mechanics of Ukrainian
National academy of sciences, Donetsk }

\bigskip

E-mail: degtyar@i.ua

\bigskip

\bigskip

\bigskip

\begin{abstract}

We consider a boundary value problem in the half-space for a linear
parabolic equation of fourth order with a degeneration on the
boundary of the half-space. The equation under consideration is
substantially a linearized  thin film equation. We prove that, if
the right hand side of the equation and the boundary condition are
polynomials in the tangential variables and time, the same property
has any solution of a power growth. It is shown also that the
specified property does not apply to normal variable. As an
application, we present a theorem of uniqueness for the problem in
the class of functions of power growth.

The final version is available on Springer at DOI: 10.1007/s00025-015-0467-x

\end{abstract}

\bigskip



\bigskip

\bigskip

\section{Introduction}

\label{s1}

The importance of the classical Liouville theorem is well known.
This theorem in complex analysis states that the entire function,
growing at infinity no more than a power, is a polynomial. In
particular, a bounded entire function is a constant. It is known
also that a similar property is inherited by solutions of many
linear and nonlinear elliptic and parabolic problems. The presence
of such property for solutions of a problem is an extremely
important tool for investigation of qualitative properties of such
solutions. As a simple example of the use of the partial Liouville
theorem of this paper (Theorem \ref{T1.1})we present a result on the
uniqueness for the initial-boundary value problems in a halfspace
for a fourth order degenerate linear equation in the class of
solutions of power growth. The literature on the subject is so vast
and diverse that we do not try to give a brief overview in this
introduction. Among the papers on the Liouville properties of
solutions for degenerate equations we mention only the papers
\cite{1} - \cite{10} and it is definitely not complete list. We note
only that all such papers for degenerate equations are devoted
mainly to second-order equations. Author does not know about the
Liouville theorems for degenerate parabolic equations of higher than
the second order.

In this paper, we consider in the halfspace a boundary value problem
for a fourth order parabolic equation with strong degeneration on
the boundary of the halfspace. This problem is obtained as a model
case under linearization of boundary value problem for the
quasilinear degenerate thin film equation in the formulation of the
papers \cite{11} - \cite{15}. Denote

$R_{+}^{N}=\{x=(x_{1},...,x_{N})\in R^{N}:\ x_{N}>0\}$,
$Q=R_{+}^{N}\times R^{1}=\{(x,t):\ x\in R_{+}^{N},t\in R^{1}\}$. For
 $R>0$ denote also $B_{R}=\left\{ x=(x^{\prime },x_{N})\in
R_{+}^{N}:\quad 0<x_{N}<R,|x^{\prime }|<R\right\} $, $x^{\prime
}=(x_{1},...,x_{N-1})$, $Q_{R}=\left\{ (x,t)\in Q:\quad x\in
B_{R},-R^{2}<t<R^{2}\right\} $.

Let a function  $u(x,t)$ be defined in $Q$ ,$u(x,t)\in W_{\infty
,loc}^{4,1}(Q)$ and let $u(x,t)$  satisfies in this domain to the
equation

\bigskip \begin{equation}
\frac{\partial u}{\partial t}+\nabla \left( x_{N}^{2}\nabla \Delta u-\beta
\nabla u\right) =f(x,t),  \label{1.1}
\end{equation}
where  $\nabla =(\partial /\partial x_{1},...,\partial /\partial
x_{N})$ , $\Delta $ is the Laplace operator in space variables,
$\beta \geq 0$ is a given nonnegative constant, $f(x,t)$ is a given
function. We suppose that the function $u(x,t)$ possesses the
following weighted regularity

\begin{equation}
\sum_{|\alpha |=4}\max_{Q_{R}}\left\vert x_{N}^{2}D_{x}^{\alpha
}u\right\vert +\sum_{|\alpha |=3}\max_{Q_{R}}\left\vert x_{N}D_{x}^{\alpha
}u\right\vert \leq C(R)<\infty ,  \label{1.2}
\end{equation}
where $\alpha =(\alpha _{1},...,\alpha _{N})$ is a multiindex,
$D_{x}^{\alpha }u=\partial ^{|\alpha |}u/\partial x_{1}^{\alpha
_{1}}...\partial x_{N}^{\alpha _{N}}$, $C(R)$ is a depending on $R$
constant. We suppose also that the function $u(x,t)$ satisfies on
the boundary of the domain $Q$, that is at $x_{N}=0$, to the
Dirichlet condition

\begin{equation}
u(x,t)|_{x_{N}=0}=u(x^{\prime },0,t)=g(x^{\prime },t),  \label{1.3}
\end{equation}
where $g(x^{\prime },t)$ is a given function.

Note here that, as it was shown in the paper \cite{11}, for example,
in the one-dimensional setting, conditions \eqref{1.2}, \eqref{1.3}
uniquely determine the solution $u(x,t)$. Thus, the restriction for
the class of solution in the form as in \eqref{1.2} serves in some
sense as a replacement for the second boundary condition, which is
necessary for uniformly parabolic fourth order equations of the form
\eqref{1.1} (in this regard see Remark \ref{R1.1} ). We also note
that condition \eqref{1.2} can actually be relaxed to, for example,
the local weighted integrability conditions for the senior
derivatives. It is not our purpose in this paper to find the exact
such conditions.

Suppose finally that the function $u(x,t)$ has power growth at
infinity

\begin{equation}
\max_{Q_{R}}|u(x,t)|\leq CR^{M},\quad R\geq 1,\quad M>0,  \label{1.4}
\end{equation}
where $C$  is some positive constant, $M$ is a given nonnegative
exponent. We agree here that in what follows we denote by the same
symbols $C$, $\nu$, all absolute constants or constants depending
only on the fixed initial data of the problem.

Let us formulate now the main result.

\begin{theorem} \label{T1.1}
Let the function $f(x,t)$ from \eqref{1.1} is a polynomial with
respect to the "tangent" variables $x^{\prime }$ and $t$ of degree
$M_{f}$, and the function $g(x^{\prime },t)$ is a polynomial of
degree $M_{g}$. \ Then, under conditions \eqref{1.1} -\eqref{1.4},
the function  $u(x,t)$ is a polynomial with respect to the variables
$x^{\prime }$ and $t$ of degree not greater than $M_{u}=[M]$.
\end{theorem}

\begin{remark} \label{R1.1}
Note that the function $u(x,t)$  is not  in general a polynomial in
the variable $x_{N}$, as it is shown by the following example.
Consider a function$v(x,t)\equiv v(x_{N})$, that depends only on the
variable  $x_{N}$ and satisfies the simplest inhomogeneous equation
\eqref{1.1}, that is

\begin{equation}
l_{\beta }v\equiv \frac{d}{dx_{N}}\left(
x_{N}^{2}\frac{d^{3}v}{dx_{N}^{3}}-\beta \frac{dv}{dx_{N}}\right) =b,\quad b=const.  \label{1.5}
\end{equation}
The direct calculation shows that for $\beta
>0$ the general solution has the form

\begin{equation}
v(x_{N})=C_{1}x_{N}^{a_{1}}+C_{2}x_{N}^{a_{2}}+C_{3}x_{N}+C_{4}-\frac{b}{%
2\beta }x_{N}^{2},  \label{1.6}
\end{equation}
where $C_{i}$, $i=1,...,4$ are arbitrary constants and the exponents
$a_{1}$, $a_{2}$ are equal to

\begin{equation}
a_{1}=-\left( \frac{1}{2}+\sqrt{\frac{1}{4}+\beta }\right) <-1,\quad
a_{2}=-\frac{1}{2}+\sqrt{\frac{1}{4}+\beta }.  \label{1.7}
\end{equation}
At the same time for $\beta =0$

\begin{equation}
v(x_{N})=C_{1}x_{N}\ln
x_{N}+C_{2}x_{N}^{2}+C_{3}x_{N}+C_{4}-\frac{b}{2}\left( x_{N}^{2}\ln x_{N}-\frac{3}{2}x_{N}^{2}\right) .  \label{1.8}
\end{equation}
In this case the conditions on the class of solutions in \eqref{1.2}
are used to determine one of the arbitrary constants in \eqref{1.6},
\eqref{1.8} (exactly here they are a replacement of the boundary
condition). It follows from \eqref{1.2} that the constant $C_{1}$ in
relations \eqref{1.6}, \eqref{1.8} should be chosen by zero since
the corresponding terms do not satisfy \eqref{1.2}. Thus,

\begin{equation}
v(x_{N})=\left\{ \begin{array}{c}
  C_{2}x_{N}^{a_{2}}+C_{3}x_{N}+C_{4}-\frac{b}{2\beta
}x_{N}^{2},\quad \beta
>0, \\
  C_{2}x_{N}^{2}+C_{3}x_{N}+C_{4}-\frac{b}{2}\left( x_{N}^{2}\ln
x_{N}-\frac{3}{2}x_{N}^{2}\right) ,\quad \beta =0.
\end{array}
\right. \label{1.9}
\end{equation}
This relation gives an example of a function, satisfying all the
conditions of Theorem \ref{T1.1} and which is not a polynomial in
the variable $x_{N}$.
\end{remark}

As a simple application of Liouville theorem  \ref{T1.1} we give a
corollary about uniqueness of the solution to the initial-boundary
value problems in a halfspace for equation \eqref{1.1} in the class
of functions with power growth.

\begin{corollary} \label{C1.1}  (Uniqueness.)

Let a function $u(x,t)$ satisfies homogeneous equation \eqref{1.1}
with $f\equiv 0$ in the halfspace $Q_{+}=Q\cap \left\{ t>0\right\}
$, conditions \eqref{1.2}, homogeneous condition \eqref{1.3} with
$g\equiv 0$, and homogeneous initial condition

\begin{equation}
u(x,0)\equiv 0,\quad t=0.  \label{1.10}
\end{equation}

If $u(x,t)$ has at most power growth at infinity  (that is, $u(x,t)$
satisfies  \eqref{1.4}), then this function is identically equal to
zero, $u(x,t)\equiv 0$.
\end{corollary}

\begin{proof}

Extent the function $u(x,t)$ to the whole domain $Q$, assuming that
it is identically zero in  $\left\{ t\leq 0\right\} $,  and save its
previous designation $u(x,t)$. In view of the homogeneity of
equation \eqref{1.1} and boundary condition \eqref{1.3}, it follows
from \eqref{1.10} that the extended function satisfies homogeneous
equation  \eqref{1.1} with $f\equiv 0$ and homogeneous condition
\eqref{1.3} with $g\equiv 0$ in the whole domain $Q$. Furthermore,
the extended function inherits power growth at infinity.
Consequently, such the function $u(x,t)$ satisfies all the
conditions of Theorem \ref{T1.1} and therefore is a polynomial in
the variable $t$ for all values of the variables $x$. Since this
polynomial is identically equal to zero in $\left\{t \leq 0
\right\}$, we see that it is identically equal to zero for all $t$.
Thus, the function $u(x,t)$ is identically equal to zero.
\end{proof}

To further content of the article is constructed as follows. In the
second section we present the proof of Theorem \ref{T1.1}. This
proof is based on local integral estimates for derivatives of the
solution $u(x,t)$ with respect to the "tangent" variables
$x^{\prime} = (x_{1},...,x_{N-1})$ and with respect to the variable
$t$ via a local integral norm of the solution itself. The proof of
the necessary local integral estimates is presented for convenience
in the third final section.

\section{The proof of Theorem  \ref{T1.1}.}

Denote for $R>0$, as before,  $B_{R}=\left\{ x=(x^{\prime
},x_{N})\in R_{+}^{N}:\quad 0<x_{N}<R,|x^{\prime }|<R\right\} $,
$x^{\prime }=(x_{1},...,x_{N-1})$, $Q_{R}=\left\{ (x,t)\in Q:\quad
x\in B_{R},-R^{2}<t<R^{2}\right\} $. The proof of Theorem \ref{T1.1}
 is based on the following lemma.

\begin{lemma} \label{L2.1}
Let the function $u(x,t)$ satisfies the conditions of Theorem
\ref{T1.1} except for power growth at infinity. Let, besides,
function $u(x,t)$ be infinitely differentiable in the variables
$x^{\prime}$, $t$ in $\overline{Q}$ and its derivatives with respect
to these variables belong to the same space as $u(x,t)$ itself.
Suppose, in addition, $f \equiv 0$ in  equation \eqref{1.1} and $g
\equiv 0$ in condition \eqref {1.3}.

Then for all $R>1$, $q>1$ the following estimates are valid

\begin{equation}
\int\limits_{Q_{R}}|\nabla u|^{2}dxdt\leq
\frac{C_{q}}{R^{2}}\int\limits_{Q_{qR}}u^{2}dxdt,\quad \int\limits_{Q_{R}}u_{t}^{2}dxdt\leq
\frac{C_{q}}{R^{4}}\int\limits_{Q_{qR}}u^{2}dxdt.  \label{2.1}
\end{equation}

\end{lemma}

The proof of this lemma will be given in the next section, but now
proceed to the proof of Theorem \ref{T1.1}.

Suppose first that the function $u(x,t)$ is infinitely
differentiable with respect to the  variables $x^{\prime}$, $t$ in
$\overline{Q}$ and its derivatives with respect to these variables
belong to the same space that $u(x,t)$ itself. And let also $f
\equiv 0$ in equation \eqref{1.1} and $g \equiv 0$ in condition
\eqref{1.3}, that is, $u(x,t)$ satisfies the conditions of Lemma
\ref{L2.1}. As the coefficients of equation \eqref{1.1} does not
depend on $x^{\prime}$, $t$, differentiating this equation with
respect to these variables and using estimates \eqref{2.1}
iteratively, we obtain by induction

\begin{equation}
\int\limits_{Q_{R}}\left\vert D_{x^{\prime }}^{\alpha }D_{t}^{\beta
}u\right\vert ^{2}dxdt\leq \frac{C(q,\alpha ,\beta )}{R^{2|\alpha |+4\beta
}}\int\limits_{Q_{q^{|\alpha |+\beta }R}}u^{2}dxdt,  \label{2.2}
\end{equation}
where $\alpha =(\alpha _{1},...,\alpha _{N-1})$ is a multiindex,
$D_{x^{\prime }}^{\alpha }=D_{x_{1}}^{\alpha
_{1}}...D_{x_{N-1}}^{\alpha _{N-1}}$.  At the same time, condition
\eqref{1.4} implies the estimate of the integral on the right side
of \eqref{2.2}

\begin{equation*}
\int\limits_{Q_{q^{|\alpha |+\beta }R}}u^{2}dxdt\leq C(q,\alpha ,\beta
)R^{2M+N+2}.
\end{equation*}

And then we obtain from \eqref{2.2}

\begin{equation*}
\int\limits_{Q_{R}}\left\vert D_{x^{\prime }}^{\alpha }D_{t}^{\beta
}u\right\vert ^{2}dxdt\leq C(q,\alpha ,\beta )R^{2M+N+2-(2|\alpha |+4\beta
)}.
\end{equation*}
Choose now in the last estimate sufficiently big $\left\vert \alpha
\right\vert $ and $\beta $. Letting $R$ to infinity, we obtain
$D_{x^{\prime }}^{\alpha }D_{t}^{\beta }u\equiv 0$ in $Q $ for
$2|\alpha |+4\beta >2M+N+2$. This means that the function $u(x,t)$
is a polynomial in the variables $x^{\prime}$, $t$ and the degree of
this polynomial $M_{u} \leq \lbrack M]$ by the estimate \eqref{1.4}.
Thus, Theorem \ref{T1.1} is proved under our additional assumptions.
Let us remove now our additional assumptions.

Let $u(x,t)$ satisfies the conditions of Theorem \ref{T1.1} and let
$\omega (x^{\prime },t)$  be a function of the class $C_{0}^{\infty
}(R^{N-1}\times R^{1})$ with the compact support  in the set
$\left\{ (x^{\prime },t):|x^{\prime }|+|t|<1\right\} $ and with the
properties (mollifier)

\begin{equation*}
\int\limits_{R^{N-1}\times R^{1}}\omega (x^{\prime },t)dx^{\prime
}dt=\int\limits_{|x^{\prime }|+|t|<1}\omega (x^{\prime },t)dx^{\prime
}dt=1.
\end{equation*}
For $\varepsilon \in (0,1)$ denote  $\omega _{\varepsilon
}(x^{\prime },t)=\varepsilon ^{-N}\omega (x^{\prime }/\varepsilon
,t/\varepsilon )$ and consider the function (th smoothing with
respect to "tangent" variables)

\begin{equation*}
u_{\varepsilon }(x,t)=\omega _{\varepsilon }\ast
u=\int\limits_{R^{N-1}\times R^{1}}\omega _{\varepsilon }(x^{\prime }-\xi
^{\prime },t-\tau )u(\xi ^{\prime },x_{N},\tau )d\xi d\tau .
\end{equation*}
It follows from well-known properties of convolution that the
function $u_{\varepsilon}(x,t)$ has the following properties.

1. \ The function $u_{\varepsilon}(x,t)$ is infinitely
differentiable in the variables $x^{\prime}$ and $t$,  its
derivatives in these variables belong to the same class as the
$u(x,t)$ and satisfy \eqref{1.2} with a constant, independent of
$\varepsilon \in (0,1)$.

2. \  Since the coefficients of equation \eqref{1.1} for $u(x,t)$ do
not depend on $x^{\prime }$ and $t$, the function $u_{\varepsilon
}(x,t)$ satisfies equation \eqref{1.1} with the right hand side
$f_{\varepsilon }(x,t)$ and boundary condition \eqref{1.3} with the
function $g_{\varepsilon }(x^{\prime },t)$, where

\[
f_{\varepsilon }(x,t)=\int\limits_{R^{N-1}\times R^{1}}\omega
_{\varepsilon }(\xi ^{\prime },\tau )f(x^{\prime }-\xi ^{\prime
},x_{N},t-\tau )d\xi d\tau ,\quad
\]

\begin{equation}
 g_{\varepsilon }(x^{\prime
},t)=\int\limits_{R^{N-1}\times R^{1}}\omega _{\varepsilon }(\xi ^{\prime
},\tau )g(x^{\prime }-\xi ^{\prime },t-\tau )d\xi d\tau .  \label{2.3}
\end{equation}
At the same time it follows from \eqref{2.3} that the functions
$f_{\varepsilon }(x,t)$ and $g_{\varepsilon }(x^{\prime },t)$ are
polynomials  in the variables $x^{\prime }$,  $t$ and the degrees of
these polynomials  coincide with those of  $f(x,t)$ and $g(x^{\prime
},t)$. So the degrees of these polynomials do not depend on
$\varepsilon \in (0,1)$.

3. \ It follows from the properties of the function $\omega
_{\varepsilon }(x^{\prime },t)$ and from the definition of the
function $u_{\varepsilon }(x,t)$ that $u_{\varepsilon }(x,t)$
 satisfies condition \eqref{1.4} with a constant $C$ and the exponent
$M$, which do not depend on $\varepsilon \in (0,1)$. Moreover,

\begin{equation*}
D_{x^{\prime }}^{\alpha }D_{t}^{\beta }u_{\varepsilon
}(x,t)=\int\limits_{R^{N-1}\times R^{1}}D_{x^{\prime }}^{\alpha
}D_{t}^{\beta }\omega _{\varepsilon }(x^{\prime }-\xi ^{\prime
},t-\tau )u(\xi ^{\prime },x_{N},\tau )d\xi d\tau
\end{equation*}
and so

\begin{equation}
\max_{Q_{R}}\left\vert D_{x^{\prime }}^{\alpha }D_{t}^{\beta }u_{\varepsilon
}(x,t)\right\vert \leq \frac{C}{\varepsilon ^{|\alpha |+\beta }}R^{M}.
\label{2.4}
\end{equation}

Suppose now that the multiindex $\alpha _{0}$ and $\beta_{0}$ are
chosen so large comparatively with the degrees $M_{f}$ and  $M_{g}$
of polynomials $f_{\varepsilon }(x,t)$ and $g_{\varepsilon
}(x^{\prime },t)$ that $D_{x^{\prime }}^{\alpha _{0}}D_{t}^{\beta
_{0}}f_{\varepsilon }(x,t)\equiv 0$, $D_{x^{\prime }}^{\alpha
_{0}}D_{t}^{\beta _{0}}g_{\varepsilon }(x^{\prime },t)\equiv 0$. \
Denote $v_{\varepsilon }(x,t)=D_{x^{\prime }}^{\alpha
_{0}}D_{t}^{\beta _{0}}u_{\varepsilon }(x,t)$. It follows from the
above mentioned properties of the function $u_{\varepsilon }(x,t)$
that $v_{\varepsilon }(x,t)$ has all the same properties that the
function $u_{\varepsilon }(x,t)$, including \eqref{2.4}. The
difference is that  $v_{\varepsilon }(x,t)$ satisfies equation
\eqref{1.1} with zero right hand side and it satisfies zero boundary
condition \eqref{1.3}. Thus, by the above, the function
$v_{\varepsilon}(x,t)$ is a polynomial in the variables $x^{\prime}$
and $t$. But then, by the definition of $v_{\varepsilon}(x,t)$, the
function $u_{\varepsilon}(x,t)$ is also a polynomial in these
variables, and its degree does not depend on  $\varepsilon \in
(0,1)$ and does not exceed $[M]$, by virtue of \eqref{2.4}.

Let $\varphi (x,t)$ be an arbitrary function from  $C_{0}^{\infty
}(Q) $ and let $m=[M]+1$. Since
$D_{x_{1}}^{m}...D_{x_{N-1}}^{m}D_{t}^{m}u_{\varepsilon }(x,t)\equiv
0$, multiplying this identity by $\varphi (x,t)$ and integrating by
parts over the domain $Q$, we obtain

\begin{equation*}
\int\limits_{Q}u_{\varepsilon
}(x,t)D_{x_{1}}^{m}...D_{x_{N-1}}^{m}D_{t}^{m}\varphi (x,t)dxdt=0.
\end{equation*}
Since $u(x,t)$\bigskip $\in L_{2,loc}(Q)$, the averaged  functions
$u_{\varepsilon }(x,t)$ tend in this space to the original function
$u(x,t)$ as $\varepsilon \rightarrow 0$, as it is well known.
Therefore, taking in the last inequality the limit as $\varepsilon
\rightarrow 0 $, we get

\begin{equation*}
\int\limits_{Q}u(x,t)D_{x_{1}}^{m}...D_{x_{N-1}}^{m}D_{t}^{m}\varphi
(x,t)dxdt=0.
\end{equation*}
As the function $\varphi (x,t)$ is arbitrary, we infer that
$D_{x_{1}}^{m}...D_{x_{N-1}}^{m}D_{t}^{m}u(x,t)\equiv 0$ in the
sense of distributions. But it is well known that this means that
the function $u(x,t)$ is a polynomial with respect to the variables
$x^{\prime }$ and  $t$. Moreover, in view of \eqref{1.4}  the degree
of this polynomial does not exceed $[M]$.


Thus Theorem \ref{T1.1} is proved under the assumption of validity
of Lemma \ref{L2.1}.

\section{Proof of Lemma \ref{L2.1}.}

In this section, we prove Lemma \ref{L2.1} and this will complete
the proof of Theorem \ref{T1.1}.

We need in the future some corollary of the following statement
(\cite{16}, Lemma 3.1).

\begin{lemma} \label{L3.1}

Let $f(t)$ be a nonnegative bounded function with the domain
$[r_{0},r_{1}]$, $r_{0}\geq 0$. Suppose that for  $r_{0}\leq t<s\leq
r_{1}$

\begin{equation*}
f(t)\leq \theta f(s)+[A(s-t)^{-a}+B],
\end{equation*}
where $A$, $B$, $a$, $\theta $ are some nonnegative constants and
$0\leq \theta <1$. Then for all $r_{0}\leq t<s\leq r_{1}$

\begin{equation*}
f(t)\leq C_{\theta ,a}[A(s-t)^{-a}+B].
\end{equation*}
\end{lemma}

We will use this lemma in the following particular case. Recall that
for $R>0$ we denote $B_{R}=\left\{ x=(x^{\prime },x_{N})\in
R_{+}^{N}:\quad 0<x_{N}<R,|x^{\prime }|<R\right\} $, $x^{\prime
}=(x_{1},...,x_{N-1})$, $Q_{R}=\left\{ (x,t)\in Q:\quad x\in
B_{R},-R^{2}<t<R^{2}\right\} $.

\begin{lemma} \label{L3.2}

Let functions $u(x,t)\geq 0$, $U(x,t)\geq 0$ are locally integrable
in $Q$ and for all $R>0$, $q\in (1,3)$ we have the estimate

\begin{equation}
\int\limits_{Q_{R}}u(x,t)dxdt\leq \varepsilon
\int\limits_{Q_{qR}}u(x,t)dxdt+\frac{C}{(q-1)^{a}R^{b}}\int%
\limits_{Q_{qR}}U(x,t)dxdt+B,  \label{3.1}
\end{equation}
where $\varepsilon \in (0,1)$, $B\geq 0$.

Then for all $R>0$, $q\in (1,3)$

\begin{equation}
\int\limits_{Q_{R}}u(x,t)dxdt\leq \frac{C_{a,\varepsilon
}}{(q-1)^{a}R^{b}}\int\limits_{Q_{qR}}U(x,t)dxdt+B.  \label{3.2}
\end{equation}
\end{lemma}

This lemma is immediate consequence of Lemma \ref{L3.1}. Keeping in
mind that $R$ and $q$ in \eqref{3.1} are arbitrary, it is enough to
consider on the interval $[1,q]$ the function
$f(t)=\int\limits_{Q_{tR}}u(x,\tau )dxd\tau $ and take into account
that for $t\in \lbrack 1,q]$ we have $\int\limits_{Q_{tR}}U(x,\tau
)dxd\tau \leq A=\int\limits_{Q_{qR}}U(x,t)dxdt$.

Note that  assertions analogous to Lemma \ref{L3.2} was implicitly
used before in the papers \cite{17}- \cite{25}, for example,.

We will use also the well known Nirenberg-Gagliardo inequality in
it's particular case  (see, for example,\cite{26} or \cite{27},
Theorem 5.2)

\bigskip \begin{equation}
\int\limits_{B_{R}}|\nabla v|^{2}dx\leq C\left(
\int\limits_{B_{R}}|D_{x}^{2}v|^{2}dx\right) ^{\frac{1}{2}}\left(
\int\limits_{B_{R}}v^{2}dx\right) ^{\frac{1}{2}},  \label{3.3}
\end{equation}
where $|D_{x}^{2}v|^{2}\equiv \sum_{|\alpha |=2}|D_{x}^{\alpha
}v|^{2}$ and the constant $C$ does not depend on $R$. This
inequality is valid for functions $v(x,t)$, that vanish on the
boundary $\partial B_{R}$ together with their first derivatives. It
is well known from the theory of Dirichlet problem for the Poisson
equation that for such functions

\begin{equation}
\int\limits_{B_{R}}|D_{x}^{2}v|^{2}dx\leq C\int\limits_{B_{R}}\left(
\Delta v\right) ^{2}dx,  \label{3.4}
\end{equation}
where the constant $C$ again does not depend on $R $. Substituting
\eqref{3.4} in \eqref{3.3} and using the Cauchy inequality with
$\varepsilon $ to estimate the product on the right hand side of
\eqref{3.3}, we obtain

\begin{equation*}
\int\limits_{B_{R}}|\nabla v|^{2}dx\leq \varepsilon
\int\limits_{B_{R}}\left( \Delta v\right) ^{2}dx+\frac{С
}{\varepsilon }\int\limits_{B_{R}}v^{2}dx.
\end{equation*}
Integrating this inequality in $t$ from $-R^{2}$ t5o $R^{2}$, we
arrive at the inequality

\begin{equation}
\int\limits_{Q_{R}}|\nabla v|^{2}dxdt\leq \varepsilon
\int\limits_{Q_{R}}\left( \Delta v\right) ^{2}dxdt+\frac{C
}{\varepsilon }\int\limits_{Q_{R}}v^{2}dxdt,\quad \varepsilon >0.
\label{3.5}
\end{equation}

We will use also the well known Hardy inequality in the form (see,
for example, \cite{26} or \cite{28}, formula (0.3))

\begin{equation*}
\int\limits_{0}^{R}w^{2}dx_{N}\leq
4\int\limits_{0}^{R}x_{N}^{2}\left( \frac{\partial w}{\partial
x_{N}}\right) ^{2}dx.
\end{equation*}
This inequality is valid for such functions $w(x,t)$ that
$w|_{x_{N}=R}=0$. Integrating this inequality in $x^{\prime }\in
\left\{ |x^{\prime }|<R\right\} $ and in $t$ from $-R^{2}$ to
$R^{2}$, we obtain

\begin{equation}
\int\limits_{Q_{R}}w^{2}dxdt\leq 4\int\limits_{Q_{R}}x_{N}^{2}\left(
\frac{\partial w}{\partial x_{N}}\right) ^{2}dxdt.  \label{3.6}
\end{equation}

Turning to the proof of Lemma \ref{L2.1}, let us agree to denote
everywhere below for brevity $C_{q}=C/(q-1)^{a}$,  where  $a$ is a
nonnegative number.

The proof of Lemma \ref{L2.1} will be obtained as a result of a
collection of local integral estimates with the using of
inequalities \eqref{3.5}, \eqref{3.6} an equation \eqref{1.1}
together with the boundary condition \eqref{1.3}. In what follows
$u(x,t)$ is some fixed function and it satisfies the conditions of
Lemma \ref{L2.1}.

Let $q\in (1,3)$ , $R>1$, and let $s>0$ be sufficiently large. Let
also $\eta (x,t)$ be such a nonnegative function of the class
$C^{\infty }(\overline{Q})$ that

\[
0\leq \eta \leq 1,\quad \eta (x,t)|_{Q_{R}}\equiv 1,\quad \eta
(x,t)|_{Q\setminus Q_{qR}}\equiv 0,\quad
\]
\begin{equation}
|D_{x}^{\alpha }D_{t}^{\beta }\eta |\leq \frac{C_{\alpha ,\beta
}}{[(q-1)R]^{|\alpha |+2\beta }}\equiv \frac{C_{q}}{R^{^{|\alpha
|+2\beta }}}.  \label{3.7}
\end{equation}
Let us agree for brevity to call such a function cut-off function
for the cylinder $Q_{R}$. Consider the function $v(x,t)=u(x,t)\eta
^{s}(x,t)$. Applying to this function inequality \eqref{3.5}, we
obtain after simple calculations with the using \eqref{3.7}

\begin{equation*}
\int\limits_{Q_{R}}|\nabla u|^{2}dxdt\leq \varepsilon
\int\limits_{Q_{qR}}\left( \Delta u\right) ^{2}dxdt+\varepsilon
\frac{C_{q}}{R^{2}}\int\limits_{Q_{qR}}|\nabla u|^{2}dxdt+\left( \frac{С
}{\varepsilon }+\frac{C_{q}}{R^{4}}\right) \int\limits_{Q_{qR}}u^{2}dxdt.
\end{equation*}
Choosing in this estimate $\varepsilon =\varepsilon
_{1}R^{2}/C_{q}$, $\varepsilon _{1}\in (0,1)$ and taking into
account that $R>1$, we arrive at the inequality

\bigskip \begin{equation*}
\int\limits_{Q_{R}}|\nabla u|^{2}dxdt\leq \varepsilon
_{1}R^{2}C\int\limits_{Q_{qR}}\left( \Delta u\right) ^{2}dxdt+\varepsilon
_{1}\int\limits_{Q_{qR}}|\nabla u|^{2}dxdt+\frac{C_{q}}{\varepsilon
_{1}R^{2}}\int\limits_{Q_{qR}}u^{2}dxdt.
\end{equation*}
Now it follows from the last inequality and Lemma \ref{L3.2} that

\begin{equation}
\int\limits_{Q_{R}}|\nabla u|^{2}dxdt\leq \varepsilon
_{1}R^{2}C\int\limits_{Q_{qR}}\left( \Delta u\right)
^{2}dxdt+\frac{C_{q}}{\varepsilon _{1}R^{2}}\int\limits_{Q_{qR}}u^{2}dxdt.  \label{3.8}
\end{equation}

Let now $\eta (x,t)$ be a cut-off function of the cylinder
$Q_{q^{2}R}$ and it identically equal to $1$  on $Q_{qR}$. Denote
$w(x,t)=\eta ^{s}(x,t)\Delta u$ and apply to this function
inequality \eqref{3.6} in the domain $Q_{q^{2}R}$. After simple
calculations with the using \eqref{3.7}, we obtain

\begin{equation*}
\int\limits_{Q_{qR}}\left( \Delta u\right) ^{2}dxdt\leq С
\int\limits_{Q_{q^{2}R}}x_{N}^{2}\left( \frac{\partial }{\partial
x_{N}}\Delta u\right)
^{2}dxdt+\frac{C}{R^{2}}\int\limits_{Q_{q^{2}R}}x_{N}^{2}\left( \Delta u\right) ^{2}dxdt.
\end{equation*}
As $q\in (1,3)$ is arbitrary, substituting this estimate in
\eqref{3.8} and denoting $q^{2}$ again by $q$, we get the inequality

\[
\int\limits_{Q_{R}}|\nabla u|^{2}dxdt\leq \varepsilon
_{1}R^{2}C\int\limits_{Q_{qR}}x_{N}^{2}\left\vert \nabla \Delta
u\right\vert ^{2}dxdt+\varepsilon
_{1}C\int\limits_{Q_{qR}}x_{N}^{2}\left( \Delta u\right) ^{2}dxdt+
\]
\begin{equation}
+\frac{C_{q}}{\varepsilon _{1}R^{2}}\int\limits_{Q_{qR}}u^{2}dxdt.
\label{3.9}
\end{equation}
This is the first from several integral inequalities we need. It is
obtained from Sobolev embeddings and we did not apply equation
\eqref{1.1}.

Now our goal is to use the equation to estimate the first and the
second terms on the right hand side of  \eqref{3.9} and also to
obtain an analogous estimate for the time derivative. At this we
will use also boundary condition \eqref{1.3} and the fact that it
follows from \eqref{1.2} that

\begin{equation}
x_{N}^{2}D_{x}^{\alpha }u|_{x_{N}=0}=0,\quad |\alpha |\leq 3;\quad
x_{N}D_{x}^{\alpha }u|_{x_{N}=0}=0,\quad |\alpha |\leq 2.  \label{3.10}
\end{equation}
Recall besides that according to the conditions of the lemma the
function $u$ is infinitely differentiable with respect to the
"tangent" variables $x^{\prime }$ and $t$ in combinations with
derivatives with respect to $x_{N}$ up to the fourth order.

Let $\eta (x,t)$ be defined in \eqref{3.7}. Multiply equation
\eqref{1.1} by the function  $u(x,t)\eta ^{s}(x,t)$ and integrate by
parts.  Taking into account \eqref{1.3}, \eqref{3.10} and the fact
that the support of $\eta (x,t)$ is a compact set, we obtain

\begin{equation*}
-s\int\limits_{Q_{qR}}u^{2}\eta _{t}\eta
^{s-1}dxdt-\int\limits_{Q_{qR}}x_{N}^{2}\nabla \Delta u\nabla u\eta
^{s}dxdt-s\int\limits_{Q_{qR}}x_{N}^{2}\nabla \Delta u\nabla \eta u\eta
^{s-1}dxdt+
\end{equation*}

\begin{equation*}
+\beta \int\limits_{Q_{qR}}\left\vert \nabla u\right\vert ^{2}\eta
^{s}dxdt+\beta s\int\limits_{Q_{qR}}\nabla u\nabla \eta u\eta ^{s-1}dxdt=0.
\end{equation*}
Integrating once more by parts in the terms with the expression
$\nabla \Delta u$ and taking all terms without the expressions
$\left( \Delta u\right) ^{2}\eta ^{s}$ and  $\beta \left\vert \nabla
u\right\vert ^{2}\eta ^{s}$ to the right hand side, we get

\begin{equation*}
\int\limits_{Q_{qR}}x_{N}^{2}\left( \Delta u\right) ^{2}\eta ^{s}dxdt+\beta
\int\limits_{Q_{qR}}\left\vert \nabla u\right\vert ^{2}\eta
^{s}dxdt=-2\int\limits_{Q_{qR}}x_{N}\Delta u\frac{\partial u}{\partial
x_{N}}\eta ^{s}dxdt-
\end{equation*}

\begin{equation}
-s\int\limits_{Q_{qR}}x_{N}^{2}\Delta u\left[ \Delta \eta \eta
^{s-1}+(s-1)(\nabla \eta )^{2}\eta ^{s-2}\right]
udxdt-2s\int\limits_{Q_{qR}}x_{N}^{2}\Delta u\nabla u\nabla \eta \eta ^{s-1}dxdt-
\label{3.11}
\end{equation}

\begin{equation*}
-2s\int\limits_{Q_{qR}}x_{N}\Delta u\frac{\partial \eta }{\partial
x_{N}}u\eta ^{s-1}dxdt+s\int\limits_{Q_{qR}}u^{2}\eta _{t}\eta ^{s-1}dxdt-
\end{equation*}

\begin{equation*}
-\beta s\int\limits_{Q_{qR}}\nabla u\nabla \eta u\eta ^{s-1}dxdt\equiv
I_{1}+I_{2}+I_{3}+I_{4}+I_{5}+I_{6}.
\end{equation*}
We estimate the integrals $I_{1}-I_{6}$ on the right hand side of
\eqref{3.11} with the help of the integral H$\ddot{o}$lder
inequality with the exponents  $p=p^{\prime }=2$ and with the help
of Cauchy's inequality with  $\varepsilon $, taking into account
properties \eqref{3.7} of the function $\eta (x,t)$. For example, we
have for $I_{4}$

\begin{equation*}
|I_{4}|\leq С \int\limits_{Q_{qR}}\left( |x_{N}\Delta u|\eta
^{s/2}\right) \left( u\left\vert \frac{\partial \eta }{\partial
x_{N}}\right\vert \eta ^{s/2-1}\right) dxdt\leq
\end{equation*}

\bigskip \begin{equation*}
\leq C\left( \int\limits_{Q_{qR}}x_{N}^{2}\left( \Delta u\right) ^{2}\eta
^{s}dxdt\right) ^{\frac{1}{2}}\left( \int\limits_{Q_{qR}}u^{2}\left\vert
\frac{\partial \eta }{\partial x_{N}}\right\vert ^{2}\eta ^{s-2}dxdt\right)
^{\frac{1}{2}}\leq
\end{equation*}

\begin{equation*}
\leq C\left( \int\limits_{Q_{qR}}x_{N}^{2}\left( \Delta u\right) ^{2}\eta
^{s}dxdt\right) ^{\frac{1}{2}}\left(
\frac{C_{q}}{R^{2}}\int\limits_{Q_{qR}}u^{2}dxdt\right) ^{\frac{1}{2}}\leq
\end{equation*}

\begin{equation}
\leq \varepsilon \int\limits_{Q_{qR}}x_{N}^{2}\left( \Delta u\right)
^{2}\eta ^{s}dxdt+\frac{C_{q}}{\varepsilon
R^{2}}\int\limits_{Q_{qR}}u^{2}dxdt.  \label{3.12}
\end{equation}

Completely analogous estimates for the other integrals together with
the fact that $x_{N}\leq CR$ on $Q_{qR}$ and with \eqref{3.11} give

\begin{equation*}
\int\limits_{Q_{qR}}x_{N}^{2}\left( \Delta u\right) ^{2}\eta ^{s}dxdt+\beta
\int\limits_{Q_{qR}}\left\vert \nabla u\right\vert ^{2}\eta ^{s}dxdt\leq
\end{equation*}

\begin{equation*}
\leq \varepsilon С \int\limits_{Q_{qR}}x_{N}^{2}\left( \Delta
u\right) ^{2}\eta ^{s}dxdt+\frac{C_{q}}{\varepsilon
}\int\limits_{Q_{qR}}\left\vert \nabla u\right\vert ^{2}dxdt+\frac{C_{q}}{\varepsilon
R^{2}}\int\limits_{Q_{qR}}u^{2}dxdt.
\end{equation*}
Choose in this estimate $\varepsilon $ such that $\varepsilon С
=1/2$ and move the first term on the right hand side to the left.
Then, taking into account properties of $\eta $, we finally obtain

\begin{equation}
\int\limits_{Q_{R}}x_{N}^{2}\left( \Delta u\right) ^{2}dxdt+\beta
\int\limits_{Q_{R}}\left\vert \nabla u\right\vert ^{2}dxdt\leq
C_{q}\int\limits_{Q_{qR}}\left\vert \nabla u\right\vert
^{2}dxdt+\frac{C_{q}}{R^{2}}\int\limits_{Q_{qR}}u^{2}dxdt.  \label{3.12+1}
\end{equation}

Turn now to the following estimate. Using the same function $\eta $
as before, multiply equation \eqref{1.1} by  $\left( \Delta u\right)
\eta ^{s}$ integrate over $Q_{R}$ and integrate by parts in the
variables $x$ in terms without $\beta $. Bearing in mind that
$u_{t}|_{x_{N}=0}=0$ in view of boundary condition \eqref{1.3}, we
have

\begin{equation*}
-\int\limits_{Q_{qR}}\nabla u_{t}\nabla u\eta
^{s}dxdt-s\int\limits_{Q_{qR}}u_{t}\nabla u\nabla \eta \eta
^{s-1}dxdt-\int\limits_{Q_{qR}}x_{N}^{2}(\nabla \Delta u)^{2}\eta ^{s}dxdt-
\end{equation*}

\begin{equation*}
-s\int\limits_{Q_{qR}}x_{N}^{2}\nabla \Delta u\nabla \eta \Delta u\eta
^{s-1}dxdt-\beta \int\limits_{Q_{qR}}(\Delta u)^{2}\eta ^{s}dxdt=0.
\end{equation*}
Substituting in the first term $\nabla u_{t}\nabla u=\left[ (\nabla
u)^{2}/2\right] _{t}$ and integrating by parts in $t$ in this term,
we can represent the last equality as

\begin{equation*}
\int\limits_{Q_{qR}}x_{N}^{2}(\nabla \Delta u)^{2}\eta ^{s}dxdt+\beta
\int\limits_{Q_{qR}}(\Delta u)^{2}\eta
^{s}dxdt=-s\int\limits_{Q_{qR}}u_{t}\nabla u\nabla \eta \eta ^{s-1}dxdt-
\end{equation*}

\begin{equation}
-s\int\limits_{Q_{qR}}x_{N}^{2}\nabla \Delta u\nabla \eta \Delta u\eta
^{s-1}dxdt+\frac{s}{2}\int\limits_{Q_{qR}}(\nabla u)^{2}\eta _{t}\eta
^{s-1}dxdt\equiv I_{1}+I_{2}+I_{3}.  \label{3.12+2}
\end{equation}

The integrals $I_{1}-I_{3}$ are estimated as before. In particular,

\begin{equation*}
|I_{1}|\leq C_{q}\int\limits_{Q_{qR}}|u_{t}|\frac{|\nabla u|}{R}dxdt\leq
\varepsilon _{2}\int\limits_{Q_{qR}}u_{t}^{2}dxdt+\frac{C_{q}}{\varepsilon
_{2}R^{2}}\int\limits_{Q_{qR}}(\nabla u)^{2}dxdt,
\end{equation*}

\begin{equation*}
|I_{2}|\leq \varepsilon \int\limits_{Q_{qR}}x_{N}^{2}(\nabla \Delta
u)^{2}\eta ^{s}dxdt+\frac{C_{q}}{\varepsilon
R^{2}}\int\limits_{Q_{qR}}x_{N}^{2}(\Delta u)^{2}dxdt,
\end{equation*}

\begin{equation*}
|I_{3}|\leq \frac{C_{q}}{R^{2}}\int\limits_{Q_{qR}}(\nabla u)^{2}dxdt.
\end{equation*}
Thus, we get from \eqref{3.12+2}

\begin{equation*}
\int\limits_{Q_{qR}}x_{N}^{2}(\nabla \Delta u)^{2}\eta ^{s}dxdt+\beta
\int\limits_{Q_{qR}}(\Delta u)^{2}\eta ^{s}dxdt\leq \varepsilon
\int\limits_{Q_{qR}}x_{N}^{2}(\nabla \Delta u)^{2}\eta ^{s}dxdt+
\end{equation*}

\begin{equation*}
+\frac{C_{q}}{\varepsilon R^{2}}\int\limits_{Q_{qR}}x_{N}^{2}(\Delta
u)^{2}dxdt+\varepsilon
_{2}\int\limits_{Q_{qR}}u_{t}^{2}dxdt+\frac{C_{q}}{\varepsilon _{2}R^{2}}\int\limits_{Q_{qR}}(\nabla u)^{2}dxdt.
\end{equation*}
The first term on the right can be moved to the left side with the
choice $\varepsilon =1/2$. As for the second term on the right, we
use \eqref{3.12+1} with $qR$ instead of $R$ to estimate it. This
gives

\begin{equation*}
\int\limits_{Q_{qR}}x_{N}^{2}(\nabla \Delta u)^{2}\eta ^{s}dxdt+\beta
\int\limits_{Q_{qR}}(\Delta u)^{2}\eta ^{s}dxdt\leq \varepsilon
_{2}\int\limits_{Q_{qR}}u_{t}^{2}dxdt+\frac{C_{q}}{\varepsilon
_{2}R^{2}}\int\limits_{Q_{qR}}(\nabla u)^{2}dxdt+
\end{equation*}

\begin{equation*}
+\frac{C_{q}}{R^{2}}\int\limits_{Q_{q^{2}R}}\left\vert \nabla u\right\vert
^{2}dxdt+\frac{C_{q}}{R^{4}}\int\limits_{Q_{q^{2}R}}u^{2}dxdt.
\end{equation*}
Now we the properties of the function $\eta $ ($\eta \equiv 1$ on
$Q_{R}$), then we estimate the integrals over $Q_{qR}$ on the right
hand side by the same integrals over $Q_{q^{2}R}$ and as $q$ is
arbitrary, we denote $q^{2}$ again by $q$. We obtain finally

\begin{equation*}
\int\limits_{Q_{R}}x_{N}^{2}(\nabla \Delta u)^{2}dxdt+\beta
\int\limits_{Q_{R}}(\Delta u)^{2}dxdt\leq
\end{equation*}

\begin{equation}
\leq \varepsilon
_{2}\int\limits_{Q_{qR}}u_{t}^{2}dxdt+\frac{C_{q}}{\varepsilon _{2}R^{2}}\int\limits_{Q_{qR}}(\nabla
u)^{2}dxdt+\frac{C_{q}}{R^{4}}\int\limits_{Q_{qR}}u^{2}dxdt.  \label{3.15}
\end{equation}

We now turn to the following estimate. Recall that the function $u$
is infinitely differentiable in the variables $(x^{\prime },t)$ and
it satisfies the boundary conditions
$u|_{x_{N}=0}=u_{t}|_{x_{N}=0}=0$. Let $\eta $ be the same function
as above.  Multiply equation \eqref{1.1} by $\Delta u_{t}\eta ^{s}$
and integrate by parts with respect to the space variables in the
first and in the second terms. We obtain

\begin{equation*}
\int\limits_{Q_{qR}}(\nabla u_{t})^{2}\eta
^{s}dxdt+\int\limits_{Q_{qR}}x_{N}^{2}(\nabla \Delta u,\nabla \Delta
u_{t})\eta ^{s}dxdt+\beta \int\limits_{Q_{qR}}\Delta u\Delta u_{t}\eta
^{s}dxdt=
\end{equation*}

\begin{equation*}
=-s\int\limits_{Q_{qR}}\nabla u_{t}\nabla \eta u_{t}\eta
^{s-1}dxdt-s\int\limits_{Q_{qR}}x_{N}^{2}(\nabla \Delta u,\nabla \eta
)\Delta u_{t}\eta ^{s-1}dxdt.
\end{equation*}
Integrating now by parts with respect to $t$ in the second and in
the third terms on the left and moving the results to the right hand
side, we obtain

\bigskip \begin{equation*}
\int\limits_{Q_{qR}}(\nabla u_{t})^{2}\eta
^{s}dxdt=-s\int\limits_{Q_{qR}}\nabla u_{t}\nabla \eta u_{t}\eta
^{s-1}dxdt-s\int\limits_{Q_{qR}}x_{N}^{2}(\nabla \Delta u,\nabla \eta
)\Delta u_{t}\eta ^{s-1}dxdt+
\end{equation*}

\begin{equation}
+\frac{1}{2}\int\limits_{Q_{qR}}x_{N}^{2}(\nabla \Delta u)^{2}\eta _{t}\eta
^{s-1}dxdt+\frac{\beta }{2}\int\limits_{Q_{qR}}\left( \Delta u\right)
^{2}\eta _{t}\eta ^{s-1}dxdt\equiv I_{1}+I_{2}+I_{3}+I_{4}.  \label{3.16}
\end{equation}
The integrals  $I_{1}-I_{4}$ are estimated in the same way as
before. This gives

\begin{equation}
|I_{1}|\leq \frac{1}{2}\int\limits_{Q_{qR}}(\nabla u_{t})^{2}\eta
^{s}dxdt+\frac{C_{q}}{R^{2}}\int\limits_{Q_{qR}}(u_{t})^{2}dxdt,  \label{3.17}
\end{equation}

\begin{equation}
|I_{2}|\leq \varepsilon _{3}\int\limits_{Q_{qR}}x_{N}^{2}(\Delta
u_{t})^{2}dxdt+\frac{C_{q}}{\varepsilon
_{3}R^{2}}\int\limits_{Q_{qR}}x_{N}^{2}(\nabla \Delta u)^{2}dxdt,  \label{3.18}
\end{equation}

\begin{equation}
|I_{3}|\leq \frac{C_{q}}{R^{2}}\int\limits_{Q_{qR}}x_{N}^{2}(\nabla \Delta
u)^{2}dxdt,  \label{3.19}
\end{equation}

\begin{equation}
|I_{4}|\leq \frac{C_{q}}{R^{2}}\int\limits_{Q_{qR}}\left( \Delta u\right)
^{2}dxdt.  \label{3.20}
\end{equation}

Consider the integral $I_{4}$. Let  $\eta _{q}$ be a function, which
is analogous to $\eta $ with the replacing $R $ with $qR$. In
particular, $\eta _{q}\equiv 1$ on $Q_{qR}$ and $\eta _{q}\equiv 0$
outside $Q_{q^{2}R}$. Using \eqref{3.6} and \eqref{3.7} for $\eta
_{q}$, we proceed with the estimate for $I_{4}$ in the following way

\begin{equation*}
|I_{4}|\leq \frac{C_{q}}{R^{2}}\int\limits_{Q_{qR}}\left( \Delta u\eta
_{q}^{s}\right) ^{2}dxdt\leq
\frac{C_{q}}{R^{2}}\int\limits_{Q_{q^{2}R}}\left( \Delta u\eta _{q}^{s}\right) ^{2}dxdt\leq
\end{equation*}

\begin{equation*}
\leq \frac{C_{q}}{R^{2}}\int\limits_{Q_{q^{2}R}}x_{N}^{2}\left( \nabla
\left( \Delta u\eta _{q}^{s}\right) \right) ^{2}dxdt\leq
\end{equation*}

\begin{equation}
\leq \frac{C_{q}}{R^{2}}\int\limits_{Q_{q^{2}R}}x_{N}^{2}\left( \nabla
\Delta u\right)
^{2}dxdt+\frac{C_{q}}{R^{4}}\int\limits_{Q_{q^{2}R}}x_{N}^{2}\left( \Delta u\right) ^{2}dxdt.  \label{3.21}
\end{equation}
We estimate now the right hand side of \eqref{3.16} with the help of
the estimates for the integrals $I_{1}-I_{4}$. In this way we move
the first integral in \eqref{3.17} to the right hand side of
\eqref{3.16} and in estimates \eqref{3.18}, \eqref{3.19}, and
\eqref{3.21} we estimate the integrals with $x_{N}^{2}\left( \nabla
\Delta u\right) ^{2}$ and with $x_{N}^{2}\left( \Delta u\right)
^{2}$ from relations \eqref{3.15} and \eqref{3.12+1}
correspondingly. We obtain

\begin{equation*}
\int\limits_{Q_{qR}}(\nabla u_{t})^{2}\eta ^{s}dxdt\leq \varepsilon
_{3}\int\limits_{Q_{q^{3}R}}x_{N}^{2}(\Delta u_{t})^{2}dxdt+
\end{equation*}

\begin{equation*}
+\frac{C_{q}}{\varepsilon
_{3}R^{2}}\int\limits_{Q_{q^{3}R}}u_{t}^{2}dxdt+\frac{C_{q}}{\varepsilon _{2}\varepsilon
_{3}R^{4}}\int\limits_{Q_{q^{3}R}}(\nabla u)^{2}dxdt+\frac{C_{q}}{R^{6}\varepsilon
_{3}}\int\limits_{Q_{q^{3}R}}u^{2}dxdt.
\end{equation*}
Or, in view of the properties of  $\eta $ and as $q$ is arbitrary,
finally

\begin{equation}
\int\limits_{Q_{R}}(\nabla u_{t})^{2}\eta ^{s}dxdt\leq \varepsilon
_{3}\int\limits_{Q_{qR}}x_{N}^{2}(\Delta u_{t})^{2}dxdt+  \label{3.22}
\end{equation}

\begin{equation*}
+\frac{C_{q}}{\varepsilon
_{3}R^{2}}\int\limits_{Q_{qR}}u_{t}^{2}dxdt+\frac{C_{q}}{\varepsilon _{2}\varepsilon _{3}R^{4}}\int\limits_{Q_{qR}}(\nabla
u)^{2}dxdt+\frac{C_{q}}{R^{6}\varepsilon _{3}}\int\limits_{Q_{qR}}u^{2}dxdt.
\end{equation*}

Let us proceed. Let $\eta $ be again the same function as before.
Multiply equation by \eqref{1.1} by $u_{t}\eta ^{s}$ and integrate
by parts with respect to the $x$-variables:

\begin{equation*}
\int\limits_{Q_{qR}}u_{t}^{2}\eta
^{s}dxdt-\int\limits_{Q_{qR}}x_{N}^{2}\nabla \Delta u\nabla u_{t}\eta
^{s}dxdt-s\int\limits_{Q_{qR}}x_{N}^{2}\nabla \Delta u\nabla \eta u_{t}\eta
^{s-1}dxdt+
\end{equation*}

\begin{equation*}
+\beta \int\limits_{Q_{qR}}\nabla u\nabla u_{t}\eta ^{s}dxdt+\beta
s\int\limits_{Q_{qR}}\nabla u\nabla \eta u_{t}\eta ^{s-1}dxdt=0.
\end{equation*}
Integrating once again by parts in the second term with respect to
the $x$-variables, we can represent this equality in the form

\begin{equation*}
\int\limits_{Q_{qR}}u_{t}^{2}\eta
^{s}dxdt+\int\limits_{Q_{qR}}x_{N}^{2}\Delta u\Delta u_{t}\eta
^{s}dxdt+\beta \int\limits_{Q_{qR}}\nabla u\nabla u_{t}\eta ^{s}dxdt=
\end{equation*}

\begin{equation*}
=s\int\limits_{Q_{qR}}x_{N}^{2}\nabla \Delta u\nabla \eta u_{t}\eta
^{s-1}dxdt-\beta s\int\limits_{Q_{qR}}\nabla u\nabla \eta u_{t}\eta
^{s-1}dxdt-2\int\limits_{Q_{qR}}x_{N}\Delta u\frac{\partial u_{t}}{\partial
x_{N}}\eta ^{s}dxdt-
\end{equation*}

\begin{equation*}
-s\int\limits_{Q_{qR}}x_{N}^{2}\Delta u\nabla u_{t}\nabla \eta \eta
^{s-1}dxdt.
\end{equation*}
Integrating by parts with respect to the $t$-variable in the second
and in the third terms on the left and moving the results to the
right, we obtain

\begin{equation*}
\int\limits_{Q_{qR}}u_{t}^{2}\eta
^{s}dxdt=s\int\limits_{Q_{qR}}x_{N}^{2}\nabla \Delta u\nabla \eta u_{t}\eta
^{s-1}dxdt-s\int\limits_{Q_{qR}}x_{N}^{2}\Delta u\nabla u_{t}\nabla \eta
\eta ^{s-1}dxdt-
\end{equation*}

\begin{equation*}
-2\int\limits_{Q_{qR}}x_{N}\Delta u\frac{\partial u_{t}}{\partial
x_{N}}\eta ^{s}dxdt+\frac{1}{2}\int\limits_{Q_{qR}}x_{N}^{2}\left( \Delta
u\right) ^{2}\eta _{t}\eta ^{s-1}dxdt-\beta s\int\limits_{Q_{qR}}\nabla
u\nabla \eta u_{t}\eta ^{s-1}dxdt+
\end{equation*}

\begin{equation}
+\frac{\beta }{2}\int\limits_{Q_{qR}}\left( \nabla u\right) ^{2}\eta
_{t}\eta ^{s-1}dxdt\equiv I_{1}+I_{2}+I_{3}+I_{4}+I_{5}+I_{6}.  \label{3.23}
\end{equation}

We estimate the integrals $I_{1}-I_{6}$ in the same way as before.
We have:

\begin{equation*}
|I_{1}|\leq \varepsilon \int\limits_{Q_{qR}}u_{t}^{2}\eta
^{s}dxdt+\frac{C_{q}}{\varepsilon }\int\limits_{Q_{qR}}x_{N}^{2}\left( \nabla \Delta
u\right) ^{2}dxdt,
\end{equation*}
where we took into account that $x_{N}\leq С R$ on $Q_{qR}$.
Choosing, for example, $\varepsilon =\frac{1}{10}$ and estimating
the second integral from inequality \eqref{3.15}, we obtain

\begin{equation}
|I_{1}|\leq \frac{1}{10}\int\limits_{Q_{qR}}u_{t}^{2}\eta
^{s}dxdt+\varepsilon
_{2}\int\limits_{Q_{q^{2}R}}u_{t}^{2}dxdt+\frac{C_{q}}{\varepsilon _{2}R^{2}}\int\limits_{Q_{q^{2}R}}(\nabla
u)^{2}dxdt+\frac{C_{q}}{R^{4}}\int\limits_{Q_{q^{2}R}}u^{2}dxdt.  \label{3.24}
\end{equation}

Further, for $I_{2}$, taking again into account that $x_{N}\leq С R$
on $Q_{qR}$, we have

\begin{equation*}
|I_{2}|\leq \theta \int\limits_{Q_{qR}}\left( \nabla u_{t}\right)
^{2}dxdt+\frac{C_{q}}{\theta }\int\limits_{Q_{qR}}x_{N}^{2}\left( \Delta u\right)
^{2}dxdt.
\end{equation*}
Choosing here $\theta =\varepsilon _{4}R^{2}$ and using
\eqref{3.22}, \eqref{3.12+1}, we obtain

\begin{equation*}
|I_{2}|\leq \frac{\varepsilon _{4}C_{q}}{\varepsilon
_{3}}\int\limits_{Q_{q^{2}R}}u_{t}^{2}dxdt+\varepsilon _{3}\varepsilon
_{4}R^{2}\int\limits_{Q_{q^{2}R}}x_{N}^{2}(\Delta u_{t})^{2}dxdt+
\end{equation*}

\begin{equation*}
+\left( \frac{\varepsilon _{4}}{\varepsilon _{2}\varepsilon
_{3}}+\frac{1}{\varepsilon _{4}}\right) \frac{C_{q}}{R^{2}}\int\limits_{Q_{q^{2}R}}(\nabla
u)^{2}dxdt+\left( \frac{\varepsilon _{4}}{\varepsilon
_{3}}+\frac{1}{\varepsilon _{4}}\right)
\frac{C_{q}}{R^{4}}\int\limits_{Q_{q^{2}R}}u^{2}dxdt,
\end{equation*}
or, choosing $\varepsilon _{4}=\varepsilon _{5}\frac{\varepsilon
_{3}}{C_{q}}$,

\begin{equation*}
|I_{2}|\leq \varepsilon
_{5}\int\limits_{Q_{q^{2}R}}u_{t}^{2}dxdt+\varepsilon _{5}\varepsilon
_{3}^{2}C_{q}R^{2}\int\limits_{Q_{q^{2}R}}x_{N}^{2}(\Delta u_{t})^{2}dxdt+
\end{equation*}

\begin{equation}
+\left( \frac{\varepsilon _{5}}{\varepsilon _{2}}+\frac{1}{\varepsilon
_{5}\varepsilon _{3}}\right)
\frac{C_{q}}{R^{2}}\int\limits_{Q_{q^{2}R}}(\nabla u)^{2}dxdt+\left( \varepsilon _{5}+\frac{1}{\varepsilon
_{5}\varepsilon _{3}}\right)
\frac{C_{q}}{R^{4}}\int\limits_{Q_{q^{2}R}}u^{2}dxdt.  \label{3.25}
\end{equation}

The integral $I_{3}$ is estimated completely analogously to $I_{2}$,
which gives

\begin{equation*}
|I_{3}|\leq \varepsilon
_{5}\int\limits_{Q_{q^{2}R}}u_{t}^{2}dxdt+\varepsilon _{5}\varepsilon
_{3}^{2}C_{q}R^{2}\int\limits_{Q_{q^{2}R}}x_{N}^{2}(\Delta u_{t})^{2}dxdt+
\end{equation*}

\begin{equation}
+\left( \frac{\varepsilon _{5}}{\varepsilon _{2}}+\frac{1}{\varepsilon
_{5}\varepsilon _{3}}\right)
\frac{C_{q}}{R^{2}}\int\limits_{Q_{q^{2}R}}(\nabla u)^{2}dxdt+\left( \varepsilon _{5}+\frac{1}{\varepsilon
_{5}\varepsilon _{3}}\right)
\frac{C_{q}}{R^{4}}\int\limits_{Q_{q^{2}R}}u^{2}dxdt.  \label{3.26}
\end{equation}

For $I_{4}$, using again \eqref{3.12+1}, we have

\begin{equation*}
|I_{4}|\leq \frac{C_{q}}{R^{2}}\int\limits_{Q_{qR}}x_{N}^{2}\left( \Delta
u\right) ^{2}dxdt\leq
\end{equation*}

\begin{equation}
\leq \frac{C_{q}}{R^{2}}\int\limits_{Q_{q^{2}R}}(\nabla
u)^{2}dxdt+\frac{C_{q}}{R^{4}}\int\limits_{Q_{q^{2}R}}u^{2}dxdt.  \label{3.27}
\end{equation}

Further,

\begin{equation}
|I_{5}|\leq \frac{1}{10}\int\limits_{Q_{qR}}u_{t}^{2}\eta
^{s}dxdt+\frac{C_{q}}{R^{2}}\int\limits_{Q_{qR}}(\nabla u)^{2}dxdt,  \label{3.28}
\end{equation}

\begin{equation}
|I_{6}|\leq \frac{C_{q}}{R^{2}}\int\limits_{Q_{qR}}(\nabla u)^{2}dxdt.
\label{3.29}
\end{equation}

We use relations \eqref{3.24}- \eqref{3.29} to estimate the right
hand side of \eqref{3.23} and move the terms with $u_{t}^{2}\eta
^{s}$ to the left. As $q$ is arbitrary, taking into account the
properties of  $\eta $, we obtain

\begin{equation*}
\int\limits_{Q_{R}}u_{t}^{2}dxdt\leq \varepsilon
_{5}\int\limits_{Q_{qR}}u_{t}^{2}dxdt+\varepsilon _{5}\varepsilon
_{3}^{2}C_{q}R^{2}\int\limits_{Q_{qR}}x_{N}^{2}(\Delta u_{t})^{2}dxdt+
\end{equation*}

\begin{equation*}
+\left( \frac{\varepsilon _{5}}{\varepsilon _{2}}+\frac{1}{\varepsilon
_{5}\varepsilon _{3}}\right)
\frac{C_{q}}{R^{2}}\int\limits_{Q_{q^{2}R}}(\nabla u)^{2}dxdt+\left( \varepsilon _{5}+\frac{1}{\varepsilon
_{5}\varepsilon _{3}}\right)
\frac{C_{q}}{R^{4}}\int\limits_{Q_{q^{2}R}}u^{2}dxdt.
\end{equation*}

Choosing here  $\varepsilon _{5}\lessdot 1$ and using Lemma
\ref{L3.2}, we arrive at the estimate

\begin{equation*}
\int\limits_{Q_{R}}u_{t}^{2}dxdt\leq \varepsilon
_{3}^{2}C_{q}R^{2}\int\limits_{Q_{qR}}x_{N}^{2}(\Delta u_{t})^{2}dxdt+
\end{equation*}

\begin{equation}
+\frac{1}{\varepsilon _{2}\varepsilon
_{3}}\frac{C_{q}}{R^{2}}\int\limits_{Q_{qR}}(\nabla u)^{2}dxdt+\left( 1+\frac{1}{\varepsilon
_{3}}\right) \frac{C_{q}}{R^{4}}\int\limits_{Q_{qR}}u^{2}dxdt.  \label{3.30}
\end{equation}

To obtain one more integral inequality multiply equation \eqref{1.1}
by  $\left[ \nabla \left( x_{N}^{2}\nabla \Delta u\right)
_{t}\right]\eta ^{s}$ with subsequent integration  and represent the
result as

\begin{equation*}
\int\limits_{Q_{R}}u_{t}\nabla \left( x_{N}^{2}\nabla \Delta
u_{t}\right) \eta ^{s}dxdt+\frac{1}{2}\int\limits_{Q_{R}}\left(
\left[ \nabla \left( x_{N}^{2}\nabla \Delta u\right) \right]
\bigskip ^{2}\right) _{t}\eta ^{s}dxdt-
\end{equation*}
\[
-\beta \int\limits_{Q_{R}}\nabla \left( x_{N}^{2}\nabla \Delta
u\right) _{t}\Delta u\eta ^{s}dxdt=0.
\]
Integrating by parts with respect to the $x$-variables in the first
and in the third integrals and integrating by parts with respect to
the $t$-variable in the second integral, we obtain

\begin{equation*}
-\int\limits_{Q_{qR}}x_{N}^{2}\nabla u_{t}\nabla \Delta u_{t}\eta
^{s}dxdt+\frac{\beta }{2}\int\limits_{Q_{qR}}x_{N}^{2}\left( \nabla \Delta u\right)
_{t}^{2}\eta ^{s}dxdt=
\end{equation*}

\begin{equation*}
=s\int\limits_{Q_{qR}}x_{N}^{2}u_{t}\nabla \Delta u_{t}\nabla \eta
\eta ^{s-1}dxdt+\frac{s}{2}\int\limits_{Q_{qR}}\left[ \nabla \left(
x_{N}^{2}\nabla \Delta u\right) \right] \bigskip ^{2}\eta _{t}\eta
^{s-1}dxdt-
\end{equation*}
\[
-\beta s\int\limits_{Q_{qR}}x_{N}^{2}\nabla \Delta u_{t}\nabla \eta
\eta ^{s-1}\Delta udxdt.
\]
We transform the resulting equation as follows. First we integrate
by parts with respect to the $x$-variables in the first term on the
left and in the first term on the right. The integral over the
surface $\left\{ x_{N}=0\right\} $ vanishes in view of condition
\eqref{1.2} and we recall that this condition is valid under
assumptions of the lemma for the function $u$ and for it's
derivative $u_{t}$. Besides, we integrate by parts with respect to
the $t$-variable in the second term on the left and in the last term
on the right. Finally, using equation \eqref{1.1}, we just replace
in the second term on the right   $\nabla \left( x_{N}^{2}\nabla
\Delta u\right) =-u_{t}+\beta \Delta u$. As the result we get the
equality

\bigskip \begin{equation*}
\int\limits_{Q_{qR}}x_{N}^{2}\left( \Delta u_{t}\right) ^{2}\eta
^{s}dxdt=-2\int\limits_{Q_{qR}}x_{N}\left( u_{t}\right) _{x_{N}}\Delta
u_{t}\eta ^{s}dxdt-s\int\limits_{Q_{qR}}x_{N}^{2}\nabla u_{t}\Delta
u_{t}\nabla \eta \eta ^{s-1}dxdt+
\end{equation*}

\begin{equation*}
+\frac{\beta s}{2}\int\limits_{Q_{qR}}x_{N}^{2}\left( \nabla \Delta
u\right) ^{2}\eta _{t}\eta
^{s-1}dxdt-2s\int\limits_{Q_{qR}}x_{N}u_{t}\Delta u_{t}\eta _{x_{N}}\eta
^{s-1}dxdt-s\int\limits_{Q_{qR}}x_{N}^{2}\nabla u_{t}\Delta u_{t}\nabla
\eta \eta ^{s-1}dxdt-
\end{equation*}

\begin{equation*}
-s\int\limits_{Q_{qR}}x_{N}^{2}u_{t}\Delta u_{t}\nabla \left( \nabla \eta
\eta ^{s-1}\right) dxdt+\frac{s}{2}\int\limits_{Q_{qR}}\left[ -u_{t}+\beta
\Delta u\right] \bigskip ^{2}\eta _{t}\eta ^{s-1}dxdt+
\end{equation*}

\begin{equation}
+\beta s\int\limits_{Q_{qR}}x_{N}^{2}\nabla \Delta u\left( \nabla \eta \eta
^{s-1}\right) _{t}\Delta udxdt+\beta s\int\limits_{Q_{qR}}x_{N}^{2}\nabla
\Delta u\nabla \eta \eta ^{s-1}\Delta u_{t}dxdt\equiv
\sum\limits_{k=1}^{9}I_{k}.  \label{3.31}
\end{equation}
The integrals $I_{k}$ are estimated according to the same schema as
above. We have

\begin{equation}
\left\vert I_{1}\right\vert +\left\vert I_{2}\right\vert +\left\vert
I_{5}\right\vert \leq \varepsilon \int\limits_{Q_{qR}}x_{N}^{2}\left(
\Delta u_{t}\right) ^{2}\eta ^{s}dxdt+\frac{C}{\varepsilon
}\int\limits_{Q_{qR}}\left( \nabla u_{t}\right) ^{2}dxdt\leq   \label{3.32}
\end{equation}

\bigskip \begin{equation*}
\leq \varepsilon \int\limits_{Q_{qR}}x_{N}^{2}\left( \Delta u_{t}\right)
^{2}\eta ^{s}dxdt+\varepsilon _{3}\frac{C_{q}}{\varepsilon
}\int\limits_{Q_{q^{2}R}}x_{N}^{2}(\Delta u_{t})^{2}dxdt+
\end{equation*}

\begin{equation*}
+\frac{C_{q}}{\varepsilon \varepsilon
_{3}R^{2}}\int\limits_{Q_{q^{2}R}}u_{t}^{2}dxdt+\frac{C_{q}}{\varepsilon \varepsilon
_{2}\varepsilon _{3}R^{4}}\int\limits_{Q_{q^{2}R}}(\nabla
u)^{2}dxdt+\frac{C_{q}}{\varepsilon R^{6}\varepsilon _{3}}\int\limits_{Q_{q^{2}R}}u^{2}dxdt,
\end{equation*}
where we made use of \eqref{3.22}. Further,

\begin{equation*}
\left\vert I_{3}\right\vert \leq
\frac{C_{q}}{R^{2}}\int\limits_{Q_{qR}}x_{N}^{2}\left( \nabla \Delta u\right) ^{2}dxdt\leq
\end{equation*}

\begin{equation}
\leq \frac{\varepsilon
_{2}C_{q}}{R^{2}}\int\limits_{Q_{q^{2}R}}u_{t}^{2}dxdt+\frac{C_{q}}{\varepsilon
_{2}R^{4}}\int\limits_{Q_{q^{2}R}}(\nabla
u)^{2}dxdt+\frac{C_{q}}{R^{6}}\int\limits_{Q_{q^{2}R}}u^{2}dxdt,  \label{3.33}
\end{equation}
where we took into account estimate \eqref{3.15}. Further,

\begin{equation}
\left\vert I_{4}\right\vert +\left\vert I_{6}\right\vert \leq \varepsilon
\int\limits_{Q_{qR}}x_{N}^{2}\left( \Delta u_{t}\right) ^{2}\eta
^{s}dxdt+\frac{C_{q}}{\varepsilon R^{2}}\int\limits_{Q_{qR}}u_{t}^{2}dxdt,
\label{3.34}
\end{equation}

\begin{equation*}
\left\vert I_{7}\right\vert \leq
\frac{C_{q}}{R^{2}}\int\limits_{Q_{qR}}u_{t}^{2}dxdt+\beta \frac{C_{q}}{R^{2}}\left( \beta
\int\limits_{Q_{qR}}\left( \Delta u\right) ^{2}dxdt\right) \leq
\end{equation*}

\begin{equation}
\leq
\frac{C_{q}}{R^{2}}\int\limits_{Q_{q^{2}R}}u_{t}^{2}dxdt+\frac{C_{q}}{\varepsilon _{2}R^{4}}\int\limits_{Q_{q^{2}R}}(\nabla
u)^{2}dxdt+\frac{C_{q}}{R^{6}}\int\limits_{Q_{q^{2}R}}u^{2}dxdt,  \label{3.35}
\end{equation}
where we made use of \eqref{3.15}. For the next integral we have

\begin{equation*}
\left\vert I_{8}\right\vert \leq \frac{C_{q}}{R^{2}}\left(
\int\limits_{Q_{qR}}x_{N}^{2}\left( \nabla \Delta u\right) ^{2}dxdt+\beta
\int\limits_{Q_{qR}}\left( \Delta u\right) ^{2}dxdt\right) \leq
\end{equation*}

\begin{equation}
\leq \frac{\varepsilon
_{2}C_{q}}{R^{2}}\int\limits_{Q_{q^{2}R}}u_{t}^{2}dxdt+\frac{C_{q}}{\varepsilon
_{2}R^{4}}\int\limits_{Q_{q^{2}R}}(\nabla
u)^{2}dxdt+\frac{C_{q}}{R^{6}}\int\limits_{Q_{q^{2}R}}u^{2}dxdt,
\label{3.36}
\end{equation}
which is analogous to \eqref{3.15}. Further, as before

\begin{equation*}
\left\vert I_{9}\right\vert \leq \varepsilon
\int\limits_{Q_{qR}}x_{N}^{2}\left( \Delta u_{t}\right) ^{2}\eta
^{s}dxdt+\frac{C_{q}}{\varepsilon R^{2}}\int\limits_{Q_{qR}}x_{N}^{2}\left( \nabla
\Delta u\right) ^{2}dxdt\leq
\end{equation*}

\begin{equation*}
\leq \varepsilon \int\limits_{Q_{qR}}x_{N}^{2}\left( \Delta u_{t}\right)
^{2}\eta ^{s}dxdt+
\end{equation*}

\begin{equation}
+\frac{\varepsilon _{2}C_{q}}{\varepsilon
R^{2}}\int\limits_{Q_{q^{2}R}}u_{t}^{2}dxdt+\frac{C_{q}}{\varepsilon _{2}\varepsilon
R^{4}}\int\limits_{Q_{q^{2}R}}(\nabla u)^{2}dxdt+\frac{C_{q}}{\varepsilon
R^{6}}\int\limits_{Q_{q^{2}R}}u^{2}dxdt.  \label{3.37}
\end{equation}

The estimate for the left hand side of \eqref{3.31} follows from
estimates \eqref{3.32}- \eqref{3.37}. We substitute these estimates
in \eqref{3.31}, then we choose sufficiently small $\varepsilon $ in
\eqref{3.32}, \eqref{3.34}, and \eqref{3.37}  and move the
corresponding integrals with  $x_{N}^{2}\left( \Delta u_{t}\right)
^{2}\eta ^{s}$ to the left hand side of \eqref{3.31}. We also
estimate the integrals over $Q_{qR}$ on the right hand side of
\eqref{3.31} by the same integrals over  $Q_{q^{2}R}$ and we denote
$q^{2}$ again by $q$ (as $q$ is arbitrary). As a result, we obtain
the estimate

\begin{equation*}
\int\limits_{Q_{R}}x_{N}^{2}\left( \Delta u_{t}\right) ^{2}dxdt\leq
\varepsilon _{3}C_{q}\int\limits_{Q_{qR}}x_{N}^{2}(\Delta u_{t})^{2}dxdt+
\end{equation*}

\begin{equation*}
+\frac{(1+\varepsilon _{2})C_{q}}{\varepsilon
_{3}R^{2}}\int\limits_{Q_{qR}}u_{t}^{2}dxdt+\frac{C_{q}}{\varepsilon _{3}\varepsilon
_{2}R^{4}}\int\limits_{Q_{qR}}(\nabla u)^{2}dxdt+\frac{C_{q}}{\varepsilon
_{3}R^{6}}\int\limits_{Q_{qR}}u^{2}dxdt.
\end{equation*}
Choosing here and in \eqref{3.30} $\varepsilon _{3}=\theta
C_{q}^{-1}$, $\theta \in (0,1/2)$ and making use of Lemma
\ref{L3.2}, we arrive at the estimate

\begin{equation*}
\int\limits_{Q_{R}}x_{N}^{2}\left( \Delta u_{t}\right) ^{2}dxdt\leq
\end{equation*}

\begin{equation}
\leq \frac{(1+\varepsilon _{2})C_{q}}{R^{2}\theta
}\int\limits_{Q_{qR}}u_{t}^{2}dxdt+\frac{C_{q}}{\varepsilon _{2}R^{4}\theta
}\int\limits_{Q_{qR}}(\nabla u)^{2}dxdt+\frac{C_{q}}{R^{6}\theta
}\int\limits_{Q_{qR}}u^{2}dxdt.  \label{3.38}
\end{equation}

Now we will combine our estimates to prove  \eqref{2.1}. Under our
choice of  $\varepsilon _{3}=\theta C_{q}^{-1}$ it follows from
\eqref{3.30} and \eqref{3.38} that for  $\varepsilon _{2}<1$

\begin{equation*}
\int\limits_{Q_{R}}u_{t}^{2}dxdt\leq \frac{\theta ^{2}}{C_{q}}R^{2}\left(
\frac{(1+\varepsilon _{2})C_{q}}{R^{2}\theta
}\int\limits_{Q_{q^{2}R}}u_{t}^{2}dxdt+\frac{C_{q}C}{\varepsilon _{2}R^{4}\theta
}\int\limits_{Q_{q^{2}R}}(\nabla u)^{2}dxdt+\frac{C_{q}}{R^{6}\theta
}\int\limits_{Q_{qR}}u^{2}dxdt\right) +
\end{equation*}

\begin{equation*}
+\frac{1}{\varepsilon _{2}\theta
}\frac{C_{q}}{R^{2}}\int\limits_{Q_{qR}}(\nabla u)^{2}dxdt+\left( 1+\frac{1}{\theta }\right)
\frac{C_{q}}{R^{4}}\int\limits_{Q_{qR}}u^{2}dxdt\leq \end{equation*}

\begin{equation*}
\leq \theta \int\limits_{Q_{q^{2}R}}u_{t}^{2}dxdt+\frac{C_{q}}{\varepsilon
_{2}\theta R^{2}}\int\limits_{Q_{q^{2}R}}(\nabla
u)^{2}dxdt+\frac{C_{q}}{\theta R^{4}}\int\limits_{Q_{q^{2}R}}u^{2}dxdt.
\end{equation*}

Denoting here $q^{2}$ again by $q$ and applying Lemma  \ref{L3.2},
we get the estimate

\begin{equation}
\int\limits_{Q_{R}}u_{t}^{2}dxdt\leq \frac{C_{q}}{\varepsilon
_{2}R^{2}}\int\limits_{Q_{qR}}(\nabla
u)^{2}dxdt+\frac{C_{q}}{R^{4}}\int\limits_{Q_{qR}}u^{2}dxdt.  \label{3.39}
\end{equation}

Substitute now estimates \eqref{3.12+1} an \eqref{3.15} in estimate
\eqref{3.9} and take into account \eqref{3.15}. We obtain

\begin{equation*}
\int\limits_{Q_{R}}|\nabla u|^{2}dxdt\leq \end{equation*}

\begin{equation*}
\leq \varepsilon _{1}R^{2}\left( \varepsilon
_{2}\int\limits_{Q_{q^{2}R}}u_{t}^{2}dxdt+\frac{C_{q}}{\varepsilon
_{2}R^{2}}\int\limits_{Q_{q^{2}R}}(\nabla
u)^{2}dxdt+\frac{C_{q}}{R^{4}}\int\limits_{Q_{q^{2}R}}u^{2}dxdt\right) +
\end{equation*}

\begin{equation*}
+\varepsilon _{1}\left( C_{q}\int\limits_{Q_{q^{2}R}}\left\vert \nabla
u\right\vert
^{2}dxdt+\frac{C_{q}}{R^{2}}\int\limits_{Q_{q^{2}R}}u^{2}dxdt\right) +\frac{C_{q}}{\varepsilon
_{1}R^{2}}\int\limits_{Q_{qR}}u^{2}dxdt\leq
\end{equation*}

\begin{equation*}
\leq \varepsilon _{1}R^{2}\varepsilon _{2}\left( \frac{C_{q}}{\varepsilon
_{2}R^{2}}\int\limits_{Q_{q^{3}R}}(\nabla
u)^{2}dxdt+\frac{C_{q}}{R^{4}}\int\limits_{Q_{q^{3}R}}u^{2}dxdt\right) +
\end{equation*}

\begin{equation*}
+\frac{\varepsilon _{1}C_{q}}{\varepsilon
_{2}}\int\limits_{Q_{q^{2}R}}(\nabla u)^{2}dxdt+\frac{C_{q}}{\varepsilon
_{1}R^{2}}\int\limits_{Q_{q^{2}R}}u^{2}dxdt\leq
\end{equation*}

\begin{equation*}
\leq \frac{\varepsilon _{1}C_{q}}{\varepsilon
_{2}}\int\limits_{Q_{q^{3}R}}(\nabla u)^{2}dxdt+\frac{C_{q}}{\varepsilon
_{1}R^{2}}\int\limits_{Q_{q^{3}R}}u^{2}dxdt.
\end{equation*}
Denoting here $q^{3}$ again by $q$, choosing $\varepsilon
_{1}=\varepsilon \varepsilon _{2}C_{q}^{-1}$, an applying Lemma
\ref{L3.2}, we obtain finally

\begin{equation*}
\int\limits_{Q_{R}}|\nabla u|^{2}dxdt\leq
\frac{C_{q}}{R^{2}}\int\limits_{Q_{qR}}u^{2}dxdt,
\end{equation*}
that is exactly the first from inequalities \eqref{2.1}. The second
from these inequalities follows now from \eqref{3.39}.

Thus, Lemma  \ref{L2.1} is proved and this finishes also the proof
of Theorem \ref{T1.1}.



\end{document}